\documentclass[11pt]{amsart}
\usepackage[mathscr]{eucal}
\usepackage{times}
\usepackage{amsmath,amssymb,latexsym,amscd}   
\usepackage{hyperref}
\hypersetup{colorlinks=true,linkcolor=blue,citecolor=red,linktocpage=true}
\usepackage{graphicx}
\usepackage[all,cmtip]{xy}
\usepackage{fancyhdr}
\usepackage{mathalfa}
\usepackage{mathrsfs}
\usepackage{upgreek}
\usepackage{multicol}
\usepackage{stmaryrd}
\usepackage{textcomp}

\DeclareMathOperator{\B}{Boy}
\DeclareMathOperator{\tors}{-tors}
\DeclareMathOperator{\Mod}{-Mod}

\DeclareMathOperator{\R}{R}

\theoremstyle{plain}
\newtheorem*{theorem*}{Theorem}
\newtheorem{thm}{Theorem}[section]
\newtheorem{cor}[thm]{Corollary}
\newtheorem{lem}[thm]{Lemma}
\newtheorem{prop}[thm]{Proposition}

\theoremstyle{definition}
\newtheorem{dfn}[thm]{Definition}
\newtheorem{obs}[thm]{Remark}
\newtheorem{ej}[thm]{Example}

\title[Boolean perspectives of idioms and the Boyle derivative]{Boolean perspectives of idioms and the Boyle derivative}

\begin{document}
\author{Jaime Castro P\'erez,\\ Mauricio Medina B\'arcenas,\\  Jos\'e R\'ios Montes,\\ Angel Zald\'ivar}

\address{Escuela de Ingenier\'ia y Ciencias, Instituto Tecnol\'ologico y de Estudios Superiores de Monterrey, Calle del Puente 222, Tlalpan, 14380, M\'exico D.F., M\'exico.}
\email{jcastrop@itesm.mx}

\address{Department of Mathematics, Chungnam National University, Yuseong-gu, Daejeon 34134, Republic of Korea.}

\email{mmedina@cnu.ac.kr}

\address{Instituto de Matem\'aticas, Universidad Nacional, Aut\'onoma de M\'exico, Area de la Investigaci\'on Cient\'ifica, Circuito Exterior, C.U., 04510, M\'exico D.F., M\'exico.}
\email{jrios@matem.unam.mx}
\email{Corresponding author: zaldivar@matem.unam.mx}

\thanks{This work was supported by the grant UNAM-DGAPA-PAPIIT IN10057.}

\maketitle

\begin{abstract}
We are concerned with the boolean or more general with the complemented properties of idioms (complete upper-continuous modular lattices). In \cite{simmonscantor} the author introduces a device which captures in some informal speaking \emph{ how far } the idiom is from be complemented, this device is the Cantor-Bendixson derivative. There exists another device that captures some boolean properties, the so-called Boyle-derivative, this derivative is an operator on the assembly (the frame of nuclei) of the idiom. The Boyle-derivative has its origins in module theory. In this investigation we produce an idiomatic analysis of the boolean properties of any idiom using the Boyle-derivative, we give conditions on a nucleus $j$ such that $[j, tp]$ is a complete boolean algebra. We also explore some properties of nuclei $j$ such that $A_{j}$ is a complemented idiom.
\end{abstract}

\section{Introduction}\label{intro}

Frames (locales, complete Heyting algebras) as algebraic analogues of topological spaces, emerge naturally in many situations. For example consider any associative ring with unit $R$ and the category of left $R$-modules, $R\Mod$. It is known that a localization of $R\Mod$ is given by a hereditary torsion class $\mathcal{T}$, that is, a class of modules closed under isomorphism, quotients, subobjects, extensions and arbitrary coproducts. All these localizations or in a more amenable way all this classes are organized as a complete lattice that results to be a frame. This frame is called $R\tors$. For years it has been seen that a decent analysis of the categorical behaviour of $R\Mod$ can be done via $R\tors$ (see \cite{golan1986torsion}).

In many other algebraic-like-situations frames appear as a manifestation of a localization process, as in a topos $\mathcal{E}$ the localizations are exactly the Lawvere-Tierney topologies on the subobject classifier $\Omega_{\mathcal{E}}$, and it is known that the former of all Lawvere-Tierney topologies constitute a frame.

The external notion of a Lawvere-Tierney topology is a \emph{nucleus} on a frame $A$, that is, a function, $j\colon A\rightarrow A$ such that 
\begin{itemize}
\item[1.] $a\leq b\Rightarrow j(a)\leq j(b)$.
\item[2.] $a\leq j(a)$.
\item[3.] $j(a\wedge b)=j(a)\wedge j(b)$.
\item[4.] $j^{2}=j$.
\end{itemize}

Denote by $N(A)$ the set of all nuclei on a frame $A$. An important result states that $N(A)$ is a frame, thus many properties of $A$ are captured by the frame of all its nuclei. In fact as frame $N(A)$ has its own frame of nuclei $N^{2}(A)$ and so on. Doing this through the ordinals we obtain the assembly tower of $A$. The idea of this tower is to control the boolean behaviour of $A$ (\cite{wilson1994assembly}), in the sense that $A$ is a complete boolean algebra if and only if $N(A)\cong A$ (Theorem 3.9 of  \cite{simmons2006assembly}). 
This extreme case also occurs in the module theoretic realm:

\[R\tors\text{ is boolean if and only if } R \text{ is a seminartian ring }\]\[ \text{ if and only if } R\tors\cong N(R\tors)\] 

These extreme cases show us that having a complete boolean algebra much of the theory is simplified or become trivial. Nevertheless, a frame $A$ is not boolean in general. We can subtract a Boolean part of $A$, this procedure is measured by the \emph{Cantor-Bendixson derivative}. The author in \cite{simmons2014cantor} and \cite{simmonshigherex} introduces the different interactions of the Cantor-Bendixson derivatives on every level in the assembly tower and the boolean consequences in $A$.
In \cite{simmonscantor} the author observes that the Cantor-Bendixson analysis can be done for more general structures called \emph{idioms}, that is, upper-continuous modular lattices. Every frame is a distributive idiom, thus idioms are a generalization of frames.

The archetypal example of an idiom is the lattice of submodules of any module, in particular the lattice of left(right) ideals of a ring. Thus the Cantor-Bendixson process can be done for this kind of lattices and in fact one can consider the frame of nuclei on a general idiom $A$. Then the Cantor-Bendixson derivate of this frame is related in somehow with the Cantor-Bendixson derivative of $A$, to see this interaction, the author in \cite{simmonsgabriel} introduces the idiomatic counterpart of the \emph{Boyle dimension on module categories} (see \cite{boyle1978large}, \cite{castro2007some}). For any idiom $A$, the Boyle derivative exists at the level of $N(A)$, that is, $\EuScript{B}oy\colon N(A)\rightarrow N(A)$ as an inflator with some extra properties.

The principal idea of this manuscript is to perform a \emph{Boyle}-analysis for idioms in the same path as in \cite{simmons2014cantor} and \cite{simmonshigherex}. We obtain conditions on an idiom $A$ to determine when $N(A)$ is a complete boolean algebra, in fact we will do this in a more general way, we will give conditions on a nucleus $j$ to have $Boy$-dimension and the frame $[j,tp]\cong N(A_{j})$ to be a complete boolean algebra.

Let us briefly describe the organization of this manuscript, Section\ref{intro} is this Introduction and Section \ref{pre} summarizes the required material for the rest of the investigation.

Section \ref{BOYID} can be understood as the idiom facet of the Boyle derivative in module categories, we will give equivalent conditions on a nucleus $j$ to have Boyle-dimension (\ref{COH}). In section \ref{BOYDIMID} we compare this dimension with the Gabriel-dimension for idioms and we generalize some results of \cite{simmonscantor}.
The last section is the idiom counterpart of the theory of \emph{spectral Grothendieck categories}, we will use several facts of the previous sections to give generalizations of the module theory realm into the idiomatic view.

\section{Preliminaries}\label{pre}

In this section we recollect the necessary material for the development of all the investigation, in particular some facts about the Boyle derivative.

We recall some of the idiom theory that we will need, first let us begin with an  example:

Given a module $M\in R\Mod$, denote by $\Lambda(M)$ the set of all submodules of $M$. It is clear that $\Lambda(M)$ constitutes a complete lattice where suprema are not unions, moreover the following distributive laws hold: 
\[N\cap (\sum X)=\sum\{N\cap K\mid K\in X\}\] 
for any $N\in\Lambda(M)$ and $X\subseteq\Lambda(M)$ directed; and
\[K\leq N\Rightarrow (K+ L)\cap N=K+(L\cap N)\] 
for any $L\in\Lambda(M)$. Thus the lattice of any module is an upper-continuous and modular lattice. This is the idea behind \emph{idioms}:

\begin{dfn}\label{Id} 
An \emph{idiom} $(A,\leq,\bigvee,\wedge, 1, 0)$ is a complete, upper-continuous, modular lattice, that is, $A$ is a complete lattice that satisfies the following distributive laws: 
\[a\wedge (\bigvee X)=\bigvee\{a\wedge x\mid x\in X\}\leqno({\rm IDL})\]
holds for all $a\in A$  and $X\subseteq A$ directed; and 
\[a\leq b\Rightarrow (a\vee c)\wedge b=a\vee(c\wedge b)\leqno({\rm ML})\]
for all $a,b,c\in A$.
These are the Idiom distributive law and the modular law respectively.
\end{dfn} 

A good account of the many uses of these lattices can be found in \cite{simmonsintroduction}. A remarkable class of idioms are the distributive ones:

\begin{dfn}\label{Id1}
A \emph{frame} $(A, \leq, \bigvee, \wedge, 1, 0)$ is a complete lattice that satisfies
\[a\wedge (\bigvee X)=\bigvee\{a\wedge x\mid x\in X\}\leqno({\rm FDL})\]
for all $a\in A$ and $X\subseteq A$ any subset.
\end{dfn} 

That is, a distributive idiom is exactly a frame. 

Frames are the algebraic version of a topological space. Indeed, if $S$ is a topological space then its topology, $\mathcal{O}(S)$ is a frame.  

There exists an important characterization of frames in terms an \emph{implication}. Recall that in any lattice $A$, an \emph{implication} in $A$ is an operation $(\_\succ \_)$ given by \[x\leq (a\succ b)\Leftrightarrow x\wedge a\leq b\], for all $a,b\in A$.  When the lattice $A$ has an implication then $A$ is a dsitributive lattice, in the context of complete lattices we have the following: 

\begin{prop}\label{03}
A complete lattice $A$ is a frame if and only if $A$ has an implication.
\end{prop}
For a proof, see Lemma 1.7 of \cite{simmons2006basics} or Theorem I4.2 of \cite{johnstone1986stone}.
We require some point-free techniques. 
\begin{dfn}\label{Id2}\begin{itemize}
\item[(1)] An \emph{inflator} on an idiom $A$ is a function $d\colon A\rightarrow A$ such that $x\leq d(x)$ and  $x\leq y \Rightarrow d(x)\leq d(y)$ for all $x,y\in A.$
\item[(2)] A \emph{pre-nucleus} $d$ on $A$ is an inflator such that $d(x\wedge y)=d(x)\wedge d(y)$ for all $x,y\in A.$ 
\item[(3)] A \emph{stable} inflator on $A$ is an inflator such that $d(x)\wedge y\leq d(x\wedge y)$ for all $x,y\in A.$
\item[(4)] A \emph{closure operator} is an idempotent inflator $c$ on $A$, that is,  is an inflator such that $c^{2}=c$.
\item[(5)] A \emph{nucleus} on $A$ is an idempotent pre-nucleus.
\end{itemize}
\end{dfn}   
Let  $I(A)$ denote the set of all inflators on $A$, $P(A)$ the set of all prenuclei, and $S(A)$ the set of all stable inflators. Clearly, $P(A)\subseteq S(A)\subseteq I(A)$. Let $C(A)$ the set of all closure operators in $A$.  Let  $N(A)$ be the set of all nuclei on $A$. All these sets are partially ordered by $d\leq f\Leftrightarrow d(a)\leq f(a)$ for all $a\in A$. Note that the identity function $id_{A}$ and the constant function $tp(a)=1$ for all $a\in A$ (where $1$ is the top of $A$) are inflators. These two inflators are the bottom and the top in all these partially ordered sets.  

Given an inflator $d\in I(A)$ we can construct a closure operator as follows:  $d^{0}:=id_{A}$,  $d^{\alpha+1}:=d\circ d^{\alpha}$ for a non-limit ordinal $\alpha$, and $d^{\lambda}:=\bigvee\{d^{\alpha}\mid \alpha<\lambda\}$ for a limit ordinal $\lambda$. These are inflators, ordered in a chain 
\[d\leq d^{2}\leq d^{3}\leq\ldots \leq d^{\alpha}\leq\ldots.\] 
By a cardinality argument, there exists an ordinal $\gamma$ such that $d^{\alpha}=d^{\gamma}$, for all $\alpha\geq \gamma$. In fact, we can choose $\gamma$ the least of these ordinals, say $\infty$. Thus, $d^{\infty}$ is an inflator such that $d\leq d^{\infty}$, but more important this inflator satisfies $d^{\infty}d^{\infty}=d^{\infty}$, that is, $d^{\infty}$ is a  closure operator on $A$. 
Also this construction gives a way to obtain nuclei on an idiom $A$.

\begin{thm}\label{Id3}
Let $A$ be an idiom then:\begin{itemize}
\item[(1)] For every stable inflator $s$ on $A$, the closure $s^{\infty}$ is a nucleus.
\item[(2)] In particular for every prenuclei $p$ on $A$, the closure $p^{\infty}$ is a nucleus.
\end{itemize}
\end{thm}


The following Theorem is one of the most important results in idiom theory.

\begin{thm}
\label{0}
For any idiom $A$, the complete lattice of all nuclei $N(A)$ on $A$ is a frame.
\end{thm}
A proof of this important fact can be found in \cite{simmmons1989near}, \cite{simmonsintroduction} (Lemma 2.4) (sometimes people call $N(A)$ the \emph{assembly} of $A$).

Any nucleus $j\in N(A)$ gives a quotient of $A$, the set $A_{j}$ of elements fixed by $j$. Even more, $A_{j}$ is an idiom, and thus many properties of $A$ are reflected in $A_{j}$ via the surjective idiom morphism (that is a monotone function that preserves-${\bigvee, \wedge, 0, 1}$ ) $j^{*}\colon A\rightarrow A_{j}$ given by $j^{*}(a)=j(a)$ for all $a\in A$. 

Now what we need is to introduce other construction of $N(A)$, to do this we need the \emph{base frame} of $A$, consider $a,b\in A$ with $a\leq b$, the \emph{interval} $[a,b]$ is the set $[a,b]=\{x\in A\mid a\leq x \leq b\}$. Denote by $\EuScript{I}(A)$ the set of all intervals of $A$. Given two intervals $I, J$, we say that $I$ is a \emph{subinterval} of $J$, if $I\subseteq J$, that is, if $I=[a,b]$ and $J=[a',b']$ with $a'\leq a\leq b\leq b'$ in $A$. We say that $J$ and $I$ are \emph{similar}, denoted by  $J\sim I$, if there are $l,r\in A$ with associated intervals \[L=[l,l\vee r]\quad [l\wedge r,r]=R\] where $J=L$ and $I=R$ or $J=R$ and $I=L$. Clearly, this a reflexive and symmetric relation. Moreover, if $A$ is modular, this relation is just the canonical lattice isomorphism between $L$ and $R$. 
\begin{dfn}\label{base}
With the same notation as above:
\begin{itemize}

\item[(1)] We say that a set of intervals $\mathcal{A}\subseteq {\EuScript I}(A)$ is \emph{abstract} if is not empty and it is closed under $\sim$, that is, 
\[J\sim I\in\mathcal{A}\Rightarrow J\in\mathcal{A}.\] 
\item[(2)] An abstract set $\mathcal{B}$ is a \emph{basic} set if it is closed under subintervals, that is, 
\[J\subseteq I\in\mathcal{B}\Rightarrow J\in\mathcal{B}.\]  
\item[(3)] A set of intervals $\mathcal{C}$ is a \emph{congruence} set if it is basic and closed under abutting intervals, that is, 
\[[a,b][b,c]\in \mathcal{C}\Rightarrow [a,c]\in\mathcal{C}\]
for elements $a,b,c\in A$. 
\item[(4)] A basic set of intervals $\mathcal{B}$ is a \emph{pre-division} set if \[\forall x\in X\left[[a,x]\in\mathcal{B}\Rightarrow [a,\bigvee X]\in\mathcal{B}\right]\] for each $a\in A$ and $X\subseteq [a,1]$. 
\item[(5)] A set of intervals $\mathcal{D}$ is a \emph{division} set if it is a congruence set and a pre-division set.
\end{itemize}
\end{dfn}

Denote $\EuScript{D}(A)\subseteq\EuScript{C}(A)\subseteq\EuScript{B}(A)\subseteq\EuScript{A}(A)$ the set of all division, congruence, basic and abstract intervals in $A$, respectively. These gadgets can be understood like certain classes of modules in a module category $R\Mod$, that is, classes closed under isomorphism, subobjects, extensions and coproducts. From this point of view $\EuScript{C}(A)$ and $\EuScript{D}(A)$ are the idiom analogues of the Serre classes and the torsion (localizations) classes in module categories.
It is not hard to see that $\EuScript{B}(A)$ is a frame, called the \emph{base} frame of the idiom $A$. The top of this frame is $\EuScript{I}(A)$ and the bottom is the set of all trivial intervals of $A$, denoted by $\EuScript{O}(A)$. Also, the family $\EuScript{C}(A)$ is a frame and a proof of this fact can be found in \cite{simmonscantor}. 

For any $\mathcal{B}\in\EuScript{B}(A)$ we can describe the least division set that contains it, denoted by $\EuScript{D}vs(\mathcal{B})$. In \cite{simmonscantor} it is proved that $\EuScript{D}vs(\_)$ is a nucleus on $\EuScript{B}(A)$ and the quotient of this nucleus is $\EuScript{D}(A)$. In fact, there is a way to connect this frame to the frame $N(A)$: 

\begin{dfn}\label{nuclei}
Fro each $a\in A$ and $\mathcal{B}\in\EuScript{B}(A)$ let $\mid\mathcal{B}\mid\colon A\rightarrow A$ be the function given by \[\mid\mathcal{B}\mid(a)=\bigvee X\]  where $x\in X\Leftrightarrow [a,x]\in\mathcal{B}$.
\end{dfn}

This produces the \emph{associated inflator} of $\mathcal{B}$. Moreover, if the basic set $\mathcal{B}$ is a congruence set, then $\mid\mathcal{B}\mid$ is a pre-nucleus on $A$, and if it is a division set, then $\mid\mathcal{B}\mid$ is a nucleus. In this way, we have for every division set  a nucleus. Now, given a nucleus $j$ we can construct a division set $\mathcal{D}_j$ as follows,
\[[a,b]\in\mathcal{D}_{j}\Leftrightarrow j(a)=j(b).\]
These correspondences are bijections and they define an isomorphism between $\EuScript{D}(A)$ and $N(A)$. With this we have:

\begin{thm}\label{00}
If  $A$ is an idiom, then there is an isomorphism of frames 
\[N(A)\longleftrightarrow \EuScript{D}(A)\qquad\qquad j\longleftrightarrow \mathcal{D}\] 
given by 
\[j\longmapsto \mathcal{D}_{j}\quad [a,b]\in\mathcal{D}_{j}\Longleftrightarrow b\leq j(a)\qquad\textrm{and}\qquad 
\mathcal{D}\longmapsto\quad j=\mid\mathcal{D}\mid.\]
\end{thm}

The $\EuScript{D}vs$-construction can be described in a useful way: 

\begin{thm}\label{000}
For every $\mathcal{B}\in\EuScript{B}(A)$ 
\[[a,b]\in\EuScript{D}vs(\mathcal{B})\Longleftrightarrow (\forall a\leq x< b)(\exists x<y\leq b)[[x,y]\in\mathcal{B}].\]
\end{thm}

The details can be found in \cite{simmonscantor} and \cite{simmonsgabriel}. This result shows that $\EuScript{D}(A)$ is a frame thus it has an implication, the following gives a description of it (for a proof see \cite{simmonscantor} Lemma 4.6).

\begin{lem}\label{div}
Let $A$ be an idiom then \[I\in(\EuScript{A}\succ\EuScript{B})\Leftrightarrow (\forall J\subseteq I)[J\in\EuScript{A}\Rightarrow J\in\EuScript{B}]\] for any $\EuScript{A}\in\EuScript{B}(A)$ and $\EuScript{B}\in\EuScript{D}(A)$.
\end{lem}

As we mentioned in the introduction we are concerned with the boolean properties of modules categories and idioms, this boolean properties are measured by some special inflators that we will introduce:

\begin{dfn}\label{I}
Let $A$ be an idiom, consider the following sets of intervals:
\begin{itemize} 
\item[(1)]An interval $[a,b]$ is \emph{simple} if there is no $a< x< b$, that is, $[a,b]=\{a,b\}$. Denote by $\EuScript{S}mp$ the set of all simple intervals.  

\item[(2)] An interval $[a,b]$ of $A$ is \emph{complemented} if it is a complemented lattice, that is, for each $a\leq x\leq b$ there exist $a\leq y\leq b$ such that $a=x\wedge y$ and  $b=x\vee y$. Let $\EuScript{C}mp$ be the set of all complemented intervals. 

\item[(3)] We can relativize this notion , for each $\mathcal{B}\in\EuScript{B}(A)$, let $\EuScript{S}mp(\mathcal{B})$ be the set of all \[[a,b]\text{ such that for each } a\leq x\leq b, [a,x]\in\mathcal{B}\text{ or } [x,b]\in\mathcal{B}\]
This is the set of $\mathcal{B}$-simple intervals.

\item[(4)] Let $\EuScript{C}mp(\mathcal{B})$ be the set of all intervals $[a,b]$ such that  :\[ \forall a\leq x\leq b \text{ exists } a\leq y\leq b\text{  such that } [a,x\wedge y]\in\mathcal{B} \text{ and } [x\vee y,b]\in\mathcal{B}.\] 
This is the set of intervals $\mathcal{B}$-complemented.
With this, we have that $\EuScript{S}mp=\EuScript{S}mp(\EuScript{O})$ and $\EuScript{C}mp=\EuScript{C}mp(\EuScript{O})$.

\item[(5)] Given any $\mathcal{B}\in\EuScript{B}(A)$ denote by $\EuScript{F}ll(\mathcal{B})$ the set of all intervals $[a,b]$ such that, \[\text{ for all } a\leq x\leq b\text{ there exists } a\leq y\leq b \text{ with } a=x\wedge y \text{ and }[x\vee y,b]\in\mathcal{B}\] this is the set of all $\mathcal{B}$-\emph{full} intervals. Note that $\EuScript{C}mp(\EuScript{O})=\EuScript{F}ll(\EuScript{O})$. 

\item[(6)] Let $\EuScript{C}rt(\mathcal{B})$ be the set of intervals $[a,b]$ such that \[\text{ for all } a\leq x\leq b \text{ we have } a=x \text{ or } [x,b]\in\mathcal{B},\] this is the set of all $\mathcal{B}$-\emph{critical} intervals. Note that $\EuScript{S}mp(\EuScript{O})=\EuScript{C}rt(\EuScript{O})$. 
\end{itemize}
\end{dfn}

In \cite{simmonsgabriel} is proved that for any $\mathcal{B}\in\EuScript{B}(A)$, $\EuScript{F}ll(\mathcal{B})\leq \EuScript{C}mp(\mathcal{B})$. Moreover, one can show that for any $\mathcal{B}\in\EuScript{B}(A)$ the sets 
$\EuScript{C}mp(\mathcal{B})$ and $\EuScript{F}ll(\mathcal{B})$
are basic.

The item 4 is the main object of study in section \ref{BOYDIMID} and section \ref{Spectralasc}
\begin{dfn}\label{Boy0}
Let $A$ be an idiom. Given elements $a,b\in A$, we say that $b$ is \emph{essentially above} $a$ \[a\lessdot b\] if $a\leq b$ and for every $y\in A$ such that \[b\wedge y\leq a\Rightarrow y\leq a\]
\end{dfn}

If the idiom is distributive, that is, a frame then this notion is equivalent to $(b\succ a)=a$, and this is the central relation of the investigation in \cite{simmonshigher} and \cite{simmons2014cantor}.

Also observe that if $a\lessdot a$ then $a=1$.
The following Lemma will be useful and a proof can be found in \cite{simmonsrelative}.
\begin{lem}\label{full}
Let $A$ be an idiom and consider any basic set $\mathcal{B}\in\EuScript{B}(A)$. For each interval $[a,b]$ the following are equivalent:\begin{itemize}
\item[(1)] $[a,b]\in\EuScript{F}ll(\mathcal{B})$.
\item[(2)] $(\forall x\in A)[a\lessdot x\Rightarrow [x\wedge b, b]\in\mathcal{B}]$
\end{itemize}
\end{lem}

$\EuScript{F}ll(\mathcal{B})$ and $\EuScript{C}rt(\mathcal{B})$ are basic sets for any basic set $\mathcal{B}$ in particular for any nucleus $j$, we can consider $\EuScript{B}oy(\EuScript{D}_{j}):=\EuScript{D}vs(\EuScript{F}ll(\EuScript{D}_{j}))$ and $\EuScript{G}ab(\EuScript{D}_{j}):=\EuScript{D}vs(\EuScript{C}rt(\EuScript{D}_{j}))$. By Theorem \ref{00} we denote the corresponding nuclei as $Boy(j)$ and $Gab(j)$ respectively. The associations $j\mapsto Boy(j)$ and $j\mapsto Gab(j)$ set up two prenuclei on $N(A)$ called the \emph{Boyle} and \emph{Gabriel} derivative respectively.


The details of these facts are not straightforward, the reader must see \cite{simmonsrelative} and \cite{simmonsgabriel}.

Let us summarize some facts about this construction

\begin{obs}\label{BOYandgab}
If we consider the simple intervals and the complement intervals as in Definition \ref{I} then we can associate the corresponding inflators, these are the \emph{socle derivative}, $soc$ and the \emph{Cantor-Bendixson derivative}, $cdb$ respectively (these two are stable inflators).
The idea is that we want the relative version of these derivatives with respect the basic set given by any nucleus $j\longleftrightarrow\EuScript{D}_{j}$, that is why the author in \cite{simmonsrelative} introduces $\EuScript{F}ll(\EuScript{D}_{j})$ and $\EuScript{C}rt(\EuScript{D}_{j})$
\begin{itemize}
\item[(1)] Let $cdb_{j}$ be the corresponding inflator of $\EuScript{F}ll(\EuScript{D}_{j})$.
\item[(2)] Let $soc_{j}$ be the corresponding inflator of $\EuScript{C}rt(\EuScript{D}_{j})$.
\item[(3)] These two inflators are stable, then by Theorem \ref{Id3} their closure $cbd_{j}^{\infty}$ and $soc_{j}^{\infty}$ are nuclei on $A$. The corresponding division sets are $\EuScript{B}oy(\EuScript{D}_{j})$ and $\EuScript{G}ab(\EuScript{D}_{j})$ respectively, thus $Boy(j)=cbd_{j}^{\infty}$ and $Gab(j)=soc_{j}^{\infty}$.
\item[(4)] Note that if $[a,b]\in\EuScript{F}ll(\EuScript{D}_{j})$ then this interval is complemented in $A_{j}$.
\end{itemize}
\end{obs}

There exists other construction for $soc_{j}$. Define the set of $j$\emph{-semicritical} intervals as the set of intervals $[a,b]$ such that there exists $X\subseteq [a,b]$ with $[a,x]\in\EuScript{C}rt(\EuScript{D}_{j})$ for each $x\in X$. Denote the set of all this intervals by $\EuScript{S}ct(\EuScript{D}_{j})$, note that, if $j=id$ then $\EuScript{S}rt(\EuScript{D}_{j})=\EuScript{S}\EuScript{S}p$ the set of \emph{ semi-simple } intervals. The set of semi-critical intervals is characterized by: \[[a,b]\in\EuScript{S}ct(\EuScript{D}_{j})\Leftrightarrow b\leq soc_{j}(a).\] Then \[soc_{j}\leftrightarrow\EuScript{S}ct(\EuScript{D}_{j})\;\;\;\; cbd_{j}\leftrightarrow\EuScript{F}ll(\EuScript{D}_{j})\]

The relation of these basic sets is described in the following:

\begin{lem}\label{0000}
For any nucleus $j$ with division set $\EuScript{D}_{j}$,
\[\EuScript{S}ct(\EuScript{D}_{j})=\EuScript{G}ab(\EuScript{D}_{j})\cap\EuScript{F}ll(\EuScript{D}_{j}).\]
\end{lem}

The proof of this Lemma is in \cite{simmonsrelative}(Lemma 6.4).

Now for last we will discuss the dimension facts about this theory.
First since $N(A)$ is a frame then it has its own inflators, in particular it has its soclo derivative $Soc$ and its Cantor-Bendixson derivative $Cbd$.

\begin{dfn}\label{dim}
Let $A$ be an idiom and $S\leq Cbd$ any stable inflator on the frame $N(A)$. For each $j\in N(A)$ we set \[LS(j)=(RS(j)\succ j)\;\;\; RS(j)=(S(j)\succ j)\] where $(\_\succ\_)$ is the implication of $N(A)$.
\end{dfn}

This two operators are studied in \cite{simmonsgabriel}. In particular the following Theorem (\cite{simmonsgabriel} Theorem 5.5) is proved.

\begin{thm}\label{dim1}
Let $A$ be an idiom and consider any stable inflator $S\leq Cbd$ on the frame $N(A)$, then \[S=S^{\infty}\wedge Cbd=LS\wedge Cbd\] in particular $S=Cbd$ if $LS=Tp$ (the top of $N^{2}(A)$).
\end{thm}

\begin{dfn}\label{dim2}
Let $A$ be an idiom and $S$ an stable inflator with $S\leq Cbd$ on the frame $N(A)$. Consider any nucleus $j\in N(A)$, we say that $j$ has:\[S\text{-dimension }\;\;\; \text{ if } S^{\infty}(j)=tp\]  and \[ \text{ weak }S\text{-dimension }\;\;\; \text{ if } LS(j)=tp.\]
\end{dfn}

It follows that if $j$ has $S$-dimension then $S=Cbd$. 

Recall that in any idiom the Cantor-Bendixson derivative $cbd$ produce the largest complemented interval $[a,cbd(a)]$ above $a$. This investigation particularizes in the Boyle-derivative so we state an important fact about it.

\begin{thm}\label{boyfacts}
Let $j$ be a nucleus on an idiom $A$.
Then the interval $[j, Boy(j)]$ is boolean.
\end{thm}
Therefore in any idiom  $Boy\leq Cbd$.
In particular if $j$ has the property that $Boy(j)=tp$ then the upper section $\uparrow(j)$ is boolean and this upper section is isomorphic to $N(A_{j})$ as frames, we will use this fact in several parts of the investigation.

\section{Boyle dimension for idioms}\label{BOYID}

First we are going to prove a slight modification of Theorem 4.10 of \cite{simmons2014cantor} and Theorems 5.13 and 5.14 of \cite{simmonshigher}. Then we will connect these ideas with the Boyle dimension. In the first Theorem is used the fact (Theorem 6.5 of \cite{simmonsrelative}) that in any idiom 
\[cdb_{j}(a)=\bigwedge\left\{j(x)\mid j(a)\lessdot x\right\}.\]

\begin{thm}\label{BOY}
For an idiom $A$ and a nucleus $j\in N(A)$ the following are equivalent:
\begin{itemize}
\item[(1)] $N(A_{j})$ is boolean.
\item[(2)] $Boy(j)=tp$.
\item[(3)] $a\lessdot cbd_{j}(a)$ for all $a\in A_{j}$ .
\end{itemize}
\end{thm}

\begin{proof}
(1) $\Leftrightarrow$ (2) It is immediate.

Suppose (2), let $a,b\in A_{j}$ such that $b\wedge cbd_{j}(a)\leq a$. By an induction argument it follows that $b\wedge cbd_{j}^{\infty}(a)\leq a$. By Remark \ref{BOYandgab} item 3, $cbd_{j}^{\infty}(a)=Boy(j)(a)$ thus $cbd_{j}^{\infty}(a)=Boy(j)(a)=1$, that is, $a\lessdot cbd_{j}(a)$. Thus, $Boy(j)(0)=cbd_{j}^\infty(0)=1$.

Conversely, if $a\lessdot cbd_{j}(a)$ for all $a\in A_{j}$ in particular for $a=cbd_{j}^{\infty}(0)$. Then $a\lessdot cbd_{j}(a)\leq cbd_{j}^{\infty}(a)=a$, therefore $a=1$.
\end{proof}

\begin{cor}\label{BOY1}
Let $A$ be an idiom then the following are equivalent:
\begin{itemize}
\item[(1)] $N(A)$ is boolean.
\item[(2)] $Boy(id)=tp$.
\item[(3)] $a\lessdot cbd(a)$ for all $a\in A$.
\end{itemize}
\end{cor}

\begin{proof}
A direct application of Theorem \ref{BOY} with $j=id$.
\end{proof}

\begin{ej}
The following lattice $A$ is the only idiom which is not a frame among all lattices with $n$ points for $1\leq n\leq 5$ (up to isomorphism).

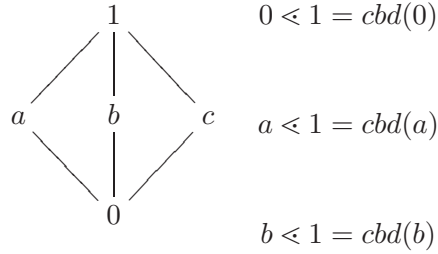
\begin{center}
\begin{tabular}{c c}
\xymatrix{ & 1 \ar@{-}[dl] \ar@{-}[dr] \ar@{-}[d] &  \\ a \ar@{-}[dr] & b \ar@{-}[d] & c \ar@{-}[dl] \\ & 0 & } &
\xymatrix{0\lessdot 1=cbd(0) \\ a\lessdot 1=cbd(a) \\ b\lessdot 1=cbd(b) \\ c\lessdot 1=cbd(c)}
\end{tabular}
\end{center}

By Corollary \ref{BOY1}, $N(A)$ is boolean. For the case $n=6$ the following idiom does not satisfies Corollary \ref{BOY1}.

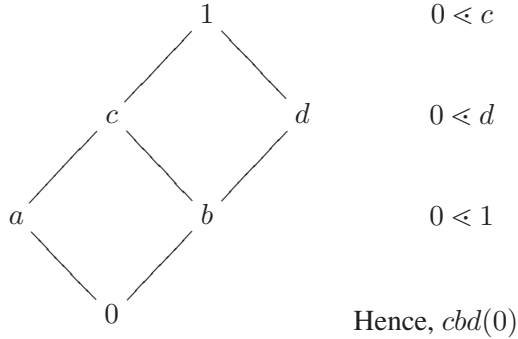
\begin{tabular}{c c}
\xymatrix{ & & 1\ar@{-}[dl] \ar@{-}[dr] & \\ & c \ar@{-}[dl] \ar@{-}[dr] & & d \ar@{-}[dl] \\ a \ar@{-}[dr] & & b \ar@{-}[dl] & \\ & 0 & & } &
\xymatrix{ 0\lessdot c \\ 0\lessdot d \\ 0\lessdot 1 \\ \text{Hence, }cbd(0)=b}
\end{tabular}

Note that $0$ is not essential below $b$.
\end{ej}

\begin{obs}\label{Boy66}
Recall that the essentially above relation \ref{Boy0} $a\lessdot b$ on a frame is equivalente to $(b\succ a)=a$, therefore observe that in particular $\R\B(j)=j\Leftrightarrow j\lessdot Boy(j)$.
\end{obs}

\begin{thm}\label{Boy7}
For an idiom $A$ the following statements are equivalent:
\begin{itemize}
\item[(1)] $N^{2}(A)$ is boolean.
\item[(2)] $\R\B(j)=j$ for all $j\in N(A)$.
\end{itemize}
\end{thm}

\begin{proof}

Just notice that if $j\lessdot Boy(j)$ then $Cbd(j)\leq Boy(j)$ from which $Cbd(j)=Boy(j)$ and in this case we can apply the argument of Theorem 4.10 in \cite{simmonscantor}. 
\end{proof}

Let $j\in N(A)$ be any nucleus. Set:
\[C_{j}=\{a\in A_j \mid a\lessdot cbd_{j}(a)\}.\] 

Observe that, if $C_{j}=A_{j}$ then $N(A_{j})$ is boolean by Theorem \ref{BOY}. In a manner of speaking $C_{j}$ measures the booleaness of the respective assembly. This boolean property is also captured by the following chain 
\[j\leq Boy(j)\leq Boy^{2}(j)\leq Boy^{3}(j)\leq\ldots\leq Boy^{\alpha}(j)\leq\] 
which eventually stabilizes in some ordinal, denote $\infty$ the minimal ordinal such that the chain stabilizes.

\begin{lem}\label{BOY8}
If $A$ is an idiom then \[C_{Boy^{\infty}(j)}=\{1\}.\]
\end{lem} 

\begin{proof}
If $a\in C_{Boy^{\infty}(j)}$ then 
\[a\lessdot cbd_{Boy^{\infty}(j)}(a)\leq Boy(Boy^{\infty}(j))(a)=Boy^{\infty}(j)(a)=a\]
thus $a=1$.
\end{proof}


With this we observe that:

\begin{prop}\label{BOY9}
With the same notation as above a nucleus $j$ has $B$-dim if and only if $C_{Boy^{\infty}(j)}=A_{Boy^{\infty}(j)}$
\end{prop}

\subsection{Cohesive properties for idioms}\label{COH}

Now we will examine the $B$-dim in idioms with ascending chain condition (ACC) on the relation $\lessdot$. To do that we use \emph{cohesive subsets} this notion is introduced in \cite{simmons2014cantor} on frames to study the second level assembly of a frame, here this notion also works fine in the idiom context.

\begin{dfn}\label{coh}
Let $A$ be an idiom. A nonempty subset $\mathcal{K}\subseteq A$ is \emph{cohesive} if for each $a\in\mathcal{K}$ there exists $X\subseteq\mathcal{K}$ such that $a=\bigwedge X$ and $a\lessdot x$ for each $x\in X$.
\end{dfn}

\begin{lem}\label{Boy1}
Suppose that $A$ has ACC on $\lessdot$. Then $\{1\}$ is the only cohesive subset.
\end{lem}

\begin{proof}
If $\mathcal{K}\subseteq A$ is cohesive and $\mathcal{K}\neq\{1\}$ then there exists $a\in\mathcal{K}$ such that $a\neq 1$ and $a=\bigwedge X$ with $a\lessdot x$ for all $x\in X$ for some $X\subseteq\mathcal{K}$. Thus $X\neq \{1\}$ that is, there is some $a'\in X$  with $a'\neq 1$ such that $a\lessdot a'$. This produce an ascending $\lessdot$-chain thus by ACC we obtain an element $b\in\mathcal{K}-\{1\}$ such that $b\lessdot b$, that is, $b=1$ a contradiction. 
\end{proof}

\begin{lem}\label{Boy2}
Let $k$ be a nucleus on $A$ such that $Boy(k)=k$. Then $A_{k}$ is cohesive.
\end{lem}

\begin{proof}
Let $a\in A_{k}$ and $X\subseteq A_{k}$ be the set of all $x$ such that $a\lessdot x$. Therefore \[a\leq\bigwedge X\leq cbd_{k}(a)\leq Boy(k)(a)=k(a)=a\]
\end{proof}

\begin{cor}\label{Boy3}
Let $A$ be an idiom and denote $k=Boy^{\infty}(j)$. Then $A_{k}$ is cohesive for any $j\in N(A)$
\end{cor}

\begin{lem}\label{Boy4}
If $\mathcal{K}$ is a cohesive subset on $A$ then \[\mathcal{K}\subseteq A_{j}\Rightarrow \mathcal{K}\subseteq A_{Boy^{\infty}(j)}\]
for each $j\in N(A)$.
\end{lem}

\begin{proof}
Let $\mathcal{K}$ be cohesive such that $\mathcal{K}\subseteq A_{j}$. Then for all $a\in \mathcal{K}$,
\[cbd_{j}(a)=\bigwedge\{x\in A_{j}\mid a\lessdot x\}\leq\bigwedge\{\ x\in  \mathcal{K}\mid a\lessdot x\}=a.\]
\end{proof}

\begin{thm}\label{Boy5}
Let $A$ be an idiom and $j\in N(A)$. The following statements hold:
\begin{itemize}
\item[(1)] $Boy^{\infty}(j)=tp\Leftrightarrow \{1\}$ is the only cohesive subset of $A_{j}$.

\item[(2)] $A_{Boy^{\infty}(j)}$ is the biggest cohesive subset of $A_{j}$.

\item[(3)] $Boy^{\infty}(j)=j$ if and only if $A_{j}$ is cohesive.
\end{itemize}

\end{thm}
\begin{proof}

(1) Let $\mathcal{K}\subseteq A_{j}$ be cohesive. It is enough to see that $\mathcal{K}\subseteq A_{Boy^{\infty}(j)}$, if $a\in K$ then \[cbd_{j}(a)=\bigwedge\{x\in A_{j}\mid a\lessdot x\}\leq\bigwedge\{x\in\mathcal{K}\mid a\lessdot x\}=a\] because the cohesive property. Now by induction one can show that $cbd_{j}^{\alpha}(a)=a$ for each ordinal $\alpha$, then $cbd_{j}^{\infty}(a)=Boy(j)(a)=a$. Again an induction argument lead to $Boy^{\alpha}(j)(a)=a$.
If $k=Boy^{\infty}(j)$ and $k=tp$ then $\mathcal{K}\subseteq A_{k}=\{1\}$.

Reciprocally if $\{1\}$ is the only cohesive subset of $A_{j}$ we have that $A_{Boy^{\infty}(j)}$ is cohesive (by Corollary \ref{Boy3}) thus $A_{Boy^{\infty}(j)}=\{1\}$, that is, $Boy^{\infty}(j)=tp$.

(2)  This is an immediate consequence of Lemma \ref{Boy4}.

(3) Put $k=Boy^\infty(j)$. If $j=k$ then $A_{j}$ is cohesive by Corollary \ref{Boy3}. 
Reciprocally if $A_{j}$ is cohesive then $A_{j}\subseteq A_{k}$ therefore $k=j$.
\end{proof}

As a consequence of Lemma \ref{Boy2} and Theorem \ref{Boy5}:
\begin{cor}\label{Boy6}
If $A_{j}$ satisfies ACC on $\lessdot$ then $j$ has $\B$-dim.
\end{cor}

All theses statements reassembles the results of \cite{simmonshigher} and the crucial fact that in a frame (a distributive idiom) the pre-nuclei $Cbd$ and $ Boy$ on $N(A)$ coincide.

\subsection{Boyle Dimension for idioms}\label{BOYDIMID}
We conclude this section with some characterizations of idioms with Boyle-dimension.

Let $j$ be a nucleus on $A$, we will give a generalization of a \emph{feebly atomic idiom}.

\begin{dfn}
An interval $[a,b]$ is \emph{feebly atomic} if for each complemented subinterval, say $a\leq c< d\leq b$ there is some $c<z\leq d$ with $[c,z]\in\EuScript{S}mp$. Let $\EuScript{F}\EuScript{A}$ denote the set of all feebly atomic intervals of $A$. The idiom $A$ is \emph{feebly atomic} if $\EuScript{F}\EuScript{A}=\EuScript{I}(A)$. Note that $\EuScript{F}\EuScript{A}$ is a division set (for details see 7.3 of \cite{simmonscantor}).
\end{dfn}

Using the fact: \[\EuScript{S}ct(\EuScript{D}_{j})=\EuScript{G}ab(\EuScript{D}_{j})\cap\EuScript{F}ll(\EuScript{D}_{j}).\]

Consider \[(\EuScript{F}ll(\EuScript{D}_{j})\succ\EuScript{S}ct(\EuScript{D}_{j}))=(\EuScript{F}ll(\EuScript{D}_{j})\succ\EuScript{G}ab(\EuScript{D}_{j}))\cap(\EuScript{F}ll(\EuScript{D}_{j})\succ\EuScript{F}ll(\EuScript{D}_{j}))=\]\[(\EuScript{F}ll(\EuScript{D}_{j})\succ\EuScript{G}ab(\EuScript{D}_{j}))\] the last equality is because in general $(a\succ a)=1$ in any frame. Since $\EuScript{G}ab(\EuScript{D}_{j})$ is a division set, the set $(\EuScript{F}ll(\EuScript{D}_{j})\succ\EuScript{G}ab(\EuScript{D}_{j}))$ is a division set.

\begin{obs}\label{obsss}
We require the following facts.
\begin{itemize}
\item[(1)] From Lemma \ref{div}, each $I\in(\EuScript{F}ll(\EuScript{D}_{j})\succ\EuScript{G}ab(\EuScript{D}_{j}))$ satisfies that any $\EuScript{D}_{j}$-full sub-interval say $[a,b]$ is in $\EuScript{G}ab(\EuScript{D}_{j})$, in other words any $j$-full sub-interval contains $j$-critical intervals.

\item[(2)] The corresponding nucleus of $(\EuScript{F}ll(\EuScript{D}_{j})\succ\EuScript{G}ab(\EuScript{D}_{j}))$, is denoted by $fbl_{j}\in N(A)$.

\end{itemize}
\end{obs}

First we prove a generalization of Theorem 7.17 of \cite{simmonscantor}.

\begin{prop}\label{latortu}
For every nucleus $j\in N(A)$ we have: \[soc_{j}=fbl_{j}\wedge cbd_{j}.\]
\end{prop}

\begin{proof}
It is know that $soc_{j}=soc_{j}^{\infty}\wedge cbd_{j}$ (Corollary 6.3 of \cite{simmonsrelative}) and by general properties of the implication on a frame we have $soc_{j}\leq soc_{j}^{\infty}\leq fbl_{j}$ then \[soc_{j}\leq fbl_{j}\wedge cbd_{j}\leq soc_{j}^{\infty}\wedge cbd_{j}=soc_{j}.\] 
\end{proof}

\begin{thm}\label{rebaba}
Let $A$ be an idiom. The following conditions are equivalent:
\begin{tabular}{l c l}
\emph{(1)} $soc_{j}=cdb_{j}$ & & \emph{(5)} $fbl_{j}=tp$ \\
\emph{(2)} $cbd_{j}\leq soc_{j}^{\infty}$ & & \emph{(6)} $\EuScript{S}ct(\EuScript{D}_{j})=\EuScript{F}ll(\EuScript{D}_{j})$ \\
\emph{(3)} $cbd_{j}\leq fbl_{j}$ & & \emph{(7)} $\EuScript{F}ll(\EuScript{D}_{j})\subseteq\EuScript{G}ab(\EuScript{D}_{j})$ \\
\emph{(4)} $cbd_{j}^{\infty}\leq soc_{j}^{\infty}$ & &  \emph{(8)} $\EuScript{F}ll(\EuScript{D}_{j})\subseteq(\EuScript{F}ll(\EuScript{D}_{j})\succ\EuScript{G}ab(\EuScript{D}_{j}))$
\end{tabular}
\end{thm}

\begin{proof}
The equivalences $(1)\Leftrightarrow (6)$, $(2)\Leftrightarrow (7)$, $(3)\Leftrightarrow (8)$ are immediate.

Now $(1)\Rightarrow (2)$ is trivial, $(2)\Rightarrow (3)$ comes from the fact that $soc_{j}^{\infty}\leq fbl_{j}$ and $(3)\Rightarrow (4)$ is clear using Proposition \ref{latortu} and the fact that $fbl_{j}$ is idempotent.

Proposition \ref{latortu} gives $(4)\Rightarrow (5)$, and is obvious that $(5)\Rightarrow (1)$.

For last, the implications $(5)\Rightarrow (6)\Rightarrow (7)$ are clear.
\end{proof}

\begin{dfn}\label{FBL}
Let $j$ be a nucleus on an idiom $A$, we say $A$ is $j$\emph{-Feebly Atomic} if satisfies the conditions of Theorem \ref{rebaba}.
\end{dfn}

It is clear that if a nucleus $j$ has Gabriel dimension then it has Boyle dimension, in a feebly atomic idiom we have a partial converse.

\begin{prop}\label{REB}
Let $j$ be a nucleus on an idiom $A$, suppose that $j$ has Boyle dimension and that $A$ is $j$-feebly atomic, then $j$ has Gabriel dimension.
\end{prop}

\begin{proof}
This is a direct consequences of Theorem \ref{rebaba} and the fact that $Boy^{\infty}(j)=tp$.
\end{proof}



One of the most important motivation in our investigation comes from ring theory and module theory. Given an associative ring $R$ with unit, let $R\Mod$ be the category of all unital left $R$-modules. There exists various ways to study $R\Mod$, a remarkable one is via its \emph{localizations}. Every localization of a Grothendieck  category (and in particular for $R\Mod$) is given by a \emph{hereditary torsion class}.
If the Gorthendieck  category is a module category, say $R\Mod$ we denote $\mathbb{D}(R)$ the set of all hereditary torsion classes. Every $\mathcal{T}\in\mathbb{D}(R)$ determines a $Hom(\mathcal{T},\_)$-orthogonal class, the \emph{torsion free class}, thus a torsion free class $\mathcal{F}$ is a class of modules closed under isomorphisms, sub-modules, products, extensions and injective hulls (denoted by $E(\_)$). The pair $\tau=(\mathcal{T},\mathcal{F})$ is called a hereditary torsion theory in $R\Mod$, denote the set of all torsion theories on $R\Mod$ by $R\tors$, observe that we can identify $\mathbb{D}(R)$ with $R\tors$. It can be seen that $R\tors$ is a frame ( Proposition 29.1 in \cite{golan1986torsion}). The book \cite{golan1986torsion} is devoted to the study of $R\Mod$ via $R\tors$.

For the definitions of the $\tau$-Gabriel dimension and $\tau$-Boyle dimension in a module category the reader is referred to  \cite{golan1986torsion}, \cite{pelaezmarcela2008}, \cite{golan1988derivatives} and \cite{castro2007some}.

\begin{thm}[\cite{pelaezmarcela2008}]\label{BOYGAB6}
Let $\tau\in R\tors$. The following conditions are equivalent:
\begin{itemize}
\item[(1)] $R$ has left $\tau$-Gabriel dimension.
\item[(2)] $R$ has left $\tau$-Boyle dimension and each $M\in\mathcal{F}_{\tau}$ contains an uniform submodule.
\end{itemize}
\end{thm}

To prove the above fact in the idiomatic realm, we need the following Lemmas. Recall that a non-trivial interval $[a,b]$ is \emph{uniform} if $x\wedge y=a$ implies either $x=a$ or $y=a$ for all $x, y\in [a,b]$.
\begin{lem}\label{lemo}
Let $j$ be a nucleus on an idiom $A$, consider the family of nuclei \[(Boy^{\alpha}(j)\mid \alpha)\] indexing over the ordinals. If $j\leq k<k'\leq Boy^{\infty}(j)$ where $k$ and $k'$ are nuclei then there exists a $\EuScript{D}_{k}$-full interval $[u,v]$ such that $[u,v]\in\EuScript{D}_{k'}$
\end{lem}

\begin{proof}

From the position of $k$ and $k'$ we can deduce that there exists an ordinal $\alpha$ such that:
\[k'\wedge Boy^{\alpha}(j)\nleq k.\] 
Let $\beta$ be the least ordinal that satisfies the condition above. Observe that  if $\beta$ is a limit ordinal we have \[k'\wedge Boy^{\beta}(j)=k'\wedge (\bigvee\{Boy^{\lambda}(j)\mid\lambda<\beta\})=\bigvee\{k'\wedge Boy^{\lambda}(j)\}\nleq k\] by definition of the chain in the limit case and the frame distributive law, thus $k'\wedge Boy^{\lambda}(j)\nleq k$ for some $\lambda<\beta$ which contradicts the choice of $\beta$, therefore $\beta$ is not a limit ordinal. 

Then we are in the situation:
\[k'\wedge Boy(Boy^{\beta-1}(j))\nleq k\] 
which on division sets it is witnessed by a interval $[a,b]$ such that $[a,b]\in\EuScript{D}_{k'\wedge Boy^{\beta}(j)}$ and $[a,b]\notin\EuScript{D}_{k}$, by the other hand we have \[[a,b]\in\EuScript{D}_{k'}\] thus $[a,b]\notin\EuScript{D}_{Boy^{\beta-1}(j)}$ then the interval $[a, Boy^{\beta-1}(j)(a)\wedge b]\in\EuScript{D}_{Boy^{\beta-1}(j)}$ and thus by Theorem \ref{00} there exists a non-trivial sub-interval $[u,v]\in\EuScript{F}ll(\EuScript{D}_{Boy^{\beta-2}(j)})$ also  $[u,v]\in\EuScript{D}_{k'}$ and then $[u,v]\in\EuScript{D}_{k}$.

Let us see that this interval is $\EuScript{D}_{k}$-full, we will use the equivalence of \ref{full}. 

Consider $x\in A$ such that $u\lessdot x$, then $[x\wedge v,v]\in\EuScript{D}_{Boy^{\beta-2}(j)}$ since $\EuScript{D}_{Boy^{\beta-2}(j)}\leq\EuScript{D}_{Boy^{\beta-1}(j)}$ and from $[u,v]\in\EuScript{D}_{k'}$ we have that $[x\wedge v,v]\in\EuScript{D}_{k'}$ also since $[u,v]\in\EuScript{D}_{k}$ we deduce $[x\wedge v, v]\in\EuScript{D}_{k}$, therefore $[u,v]\in\EuScript{F}ll(\EuScript{D}_{k})$.

\end{proof}

\begin{lem}\label{lemi}
Let $A$ be an idiom, consider any basic set $\EuScript{B}$ on $A$, suppose that we have an uniform interval $[a,b]$ which is $\EuScript{B}$-full. Then $[a,b]\in\EuScript{C}rt(\EuScript{B})$.
\end{lem}

\begin{proof}
Let $[a,b]\in\EuScript{F}ll(\EuScript{B})$ be uniform and consider any $a\leq x\leq b$. Then there exists $a\leq y\leq b$ such that $a=x\wedge y$ and $[x\vee y, b]\in\EuScript{B}$ since $[a,b]$ is $\EuScript{B}$-full. From $a=x\wedge y$ we have $a=x$ or $a=y$ because $[a,b]$ is uniform. Therefore $a=x$ or $[x,b]\in\EuScript{B}$, that is, $[a,b]\in\EuScript{C}rt(\EuScript{B})$.

\end{proof}

With these Lemmas we can give a proof of Theorem \ref{BOYGAB6} in the idiomatic context.
\begin{thm}\label{REBE1}
For a nucleus $j$ on an idiom $A$ the following conditions are equivalent:
\begin{itemize}
\item[(1)]
$j$ has Gabriel dimension.
\item[(2)]
$j$ has Boyle dimension and every interval $[a,b]$ contains a uniform sub-interval.
\end{itemize}
\end{thm}
\begin{proof}
$(1)\Rightarrow (2)$ It is immediate.

$(2)\Rightarrow (1)$ Suppose that $j$ has $B-dim=\alpha$ then $j<Boy^{\alpha}(j)=tp$. By Lemma \ref{lemo} there exists a $\EuScript{D}_{j}$-full interval containing a proper uniform interval $[a,b]$ which is $\EuScript{D}_{j}$-full. By Lemma \ref{lemi} $[a,b]$ is $\EuScript{D}_{j}$-critical thus $j$ has Gabriel dimension.

\end{proof}

\section{Spectral aspects of idioms}\label{Spectralasc}
In this section we develop a fragment of boolean aspects in the idiom theory. We introduce the concept of \emph{spectral nucleus} and then we mimic some spectral-Grothendieck  situations into the idiomatic shape. We observe that the idiomatic facet of these objects is the external version of the Grothendieck case in particular the module category realm.
First let us recall the definition of spectral category.

\begin{dfn}\label{spec0}

A Grothendieck  category $\EuScript{C}$ is \emph{spectral} if every short exact sequence  in $\EuScript{C}$ splits.
\end{dfn}

From this point, \emph{spectral category} will mean a Grothendieck category which is spectral.

Spectral categories are related with von Neumann regular rings, see \cite[V.6.1]{stenstrom1975rings} and \cite{roos1967locally}.


Before we give the definition of spectral nucleus, we will point out a motivational situation.

\begin{dfn}
Let $R$ be a ring. A nucleus $j\colon\Lambda(R)\rightarrow\Lambda(R)$ is \emph{respectful or linear} if 
\[j(I:r)=(j(I):r)\] for each $I\in\Lambda(R)$ and each $r\in R$. Denote by $\Xi(R)$ the set of linear nuclei on $\Lambda(R)$.
\end{dfn}

\begin{dfn}\label{closure}
A global closure operator on $R\Mod$ is a family $(j_M:\Lambda(M)\to \Lambda(M)\mid M\in R\Mod)$ such that every $j_M$ is a nucleus and for every morphism $f:M\to N$, 
\[f^{-1}j_N=j_Mf^{-1}:\Lambda(N)\to \Lambda(M)\]
\end{dfn}

\begin{thm}\label{lol}
Let $R$ be a ring. There is a bijective correspondence between:
\begin{itemize}
\item[(1)] $R\tors$.
\item[(2)] $\Xi(R)$.
\item[(3)] Global closure operator on $R\Mod$.
\item[(4)] Gabriel Filters of $R$.
\item[(5)] Left exact radicals of $R\Mod$.
\end{itemize}
\end{thm}

A proof of this Theorem can be found in \cite{simmonsGtop} and \cite{simmons1988semiring}.  Recall that any localization of the category of $R\Mod$ is given by an element of the frame $R\tors$, hence by Theorem \ref{lol} an element of $\Xi(R)$ is called a \emph{localizer}.

Let us recall the assignations:
\[R\tors\longleftarrow\Xi(R)\longrightarrow\text{Global operators on }R\Mod.\]

Let $j\in\Xi(R)$. 

For each module $M$ set:
\[m\in j_{M}(K)\Leftrightarrow j(K:m)=R\] 
for every $K\in\Lambda(M)$. 

This determines a nucleus on $\Lambda(M)$ and the collection $(j_{M}\mid M\in R\Mod)$ constitutes a global closure operator, in fact $j_{R}=j$.

On the other hand, $j$ defines a torsion class as follows:
\[M\in\mathcal{T}_{j}\Leftrightarrow j_{M}(0)=M.\] 


 
\begin{dfn}
Given a class of modules $\mathcal{B}$ and a module $M$, the {\it slice} of $\mathcal{B}$ by $M$, $\langle M\rangle(\mathcal{B})$ is defined as: 
\[[K,L]\in\langle M\rangle(\mathcal{B})\Leftrightarrow L/K\in\mathcal{B}.\]
\end{dfn} 

It can be seen that if $\mathcal{D}$ is a hereditary torsion class then $ \left\langle M\right\rangle(\mathcal{D})$ is division set in $\Lambda(M)$. \cite{simmonslattice}



We adopt the following definition of spectral torsion theory.
\begin{dfn}\label{espmod}
A hereditary torsion theory $\tau=(\mathcal{T},\mathcal{F})$ \emph{spectral} on a module category $R\Mod$, if the quotient category $(R, \tau)\Mod$ is a spectral category.
\end{dfn}

 Then one can proof the following:

\begin{prop}\label{spectra}
Let $\tau=(\mathcal{T},\mathcal{F})$ be a spectral torsion theory in $R\Mod$. Then for any module $M$ the idiom of sub-objects of $M$ in the quotient category is a complemented idiom.
\end{prop}

\begin{lem}\label{hl}
Let $j_{\bullet}$ be a global inflator. Then 
\[j_{M/N}\left(\frac{A}{N}\right)1=\frac{j_{M}(A)}{N}\] 
for each $N\subseteq A\subseteq M$
\end{lem}

As we mentioned before spectral aspects of Grothendieck categories give rise to certain boolean aspects, in the case of module categories this has been explored in \cite{gomez1985spectral}, \cite{arroyo1994some} and \cite{ose1997spectral}.

\begin{thm}\label{spt1}
Let $j$ be a linear nucleus on $\Lambda(R)$. Then, $\Lambda(R)_{j}$ is a complemented idiom if and only if the hereditary torsion theory with torsion class $\mathcal{T}_{j}$ is spectral.
\end{thm}

\begin{proof}
Suppose $\Lambda(R)_{j}$ is a complemented idiom.
Let $\tau_{j}$ be the hereditary torsion theory with hereditary torsion class $\mathcal{T}_{j}$ and $M$ a $\tau_{j}$-torsion free module. Let $N\in\Lambda(M)$ be an essential element. By Proposition 1.1 of \cite{arroyo1994some} it is enough to prove that $M/N\in\mathcal{T}_{j}$. For, we want to prove (using Lemma \ref{hl}) that 
\[j_{\frac{M}{N}}(0)=M/N=\frac{j_{M}(N)}{N}\] 
Let $m\in M$. By hypothesis there exists  $K\in\Lambda(R)_{j}$ such that $K\cap j(N:m)=0$ and $K\vee j(N:m)=R$. Since $N$ is essential in $M$ we have that $(N:m)$ is an essential left ideal and from $(N:m)\cap K\leq j(N:m)\cap K=0$ thus $K=0$. Therefore $j(N:m)=R$, this is equivalent to say $M=j_{M}(N)$ precisely when $M/N\in\mathcal{T}_{j}$.
The converse follows directly from the fact that for this torsion theory $\mathcal{T}_{j}$ the induced nuclei $j_{M}$ in every $\Lambda(M)$ and the corresponding quotient $\Lambda(M)_{j_{M}}$ (which is the idiom of sub-objects of every localizing object on $\mathcal{T}_{j}$) is complemented by Proposition \ref{spectra}, in particular $\Lambda(R)_{j=j_{R}}$ is complemented.
\end{proof}

Theorem \ref{spt1} motivates the following definition:

\begin{dfn}\label{sptid}
Let $A$ be an idiom. A nucleus $j$ on $A$ is \emph{spectral} if $A_{j}$ is a complemented idiom.
\end{dfn}

\begin{obs}\label{obspec}
Let $j$ be an spectral nucleus.

\begin{itemize}
\item[(1)]  This is equivalent to $cbd^{A_{j}}=tp$.
\item[(2)] From Theorem 5.8 of \cite{simmonsrelative} we find out that $Boy(j)=cbd_{j}^{\infty}=tp$ thus $j$ has weak $Boy$-dimension.
\item[(3)]Trivially every spectral nucleus satisfies Theorem \ref{BOY}. The frame $\uparrow(j)$ is a complete boolean algebra, this is straightforward since $[j, Boy(j)]$ is complemented.
\item[(4)] From $soc_{j}=soc_{j}^{\infty}\wedge cbd_{j}$ (Corollary 6.3 of \cite{simmonsrelative}) we deduce that $soc[j]=soc[j]^{\infty}$.
\item[(5)] Let $\EuScript{E}(A)$ the set of all spectral nuclei of $A$, observe that is an upper section of $N(A)$. 
\end{itemize}
\end{obs}

The conclusion of Theorem \ref{spt1} in particular one of the equivalent conditions of Proposition 1.1 of \cite{arroyo1994some} or Proposition 2.2 of \cite{ose1997spectral} implies that for spectral torsion theories, the torsion free modules are full modules. We will prove this fact in the idiomatic case.
Recall that for any interval $[a,b]$, $\chi(a,b)$ denotes the nucleus on $A$ given by $j\leq\chi(a,b)\Leftrightarrow j(a)\wedge b=a$. This is the idiomatic analogue of the cogenerated torsion theory for a module (see \cite{simmons2010decomposition} for details and uses of these nuclei).

\begin{prop}\label{lolyu}
Let $A$ be an idiom and $j\in N(A)$ a nucleus. The following statements are equivalent:
\begin{itemize}
\item[(1)] $j$ is spectral.
\item[(2)] For all $[a,b]$ with $j\leq\chi(a,b)$ we have $[a,b]\in\EuScript{F}ll(\EuScript{D}_{j})$.
\end{itemize}
\end{prop}

\begin{proof}
Suppose (1) we will use Lemma \ref{full}. Let $[a,b]$ be a non-trivial interval with $j\leq\chi(a,b)$ and let $x\in A$ be such that $a\lessdot x$. From $a\leq b\wedge x\leq b$ we have that $\chi(a,b\wedge x)(a)\wedge b\wedge x=a$ thus $\chi(a,b\wedge x)(a)=a$ (since $b\wedge x$ is essential in $[a,b]$). Therefore $j(a)=a$ and then $a\lessdot j(x)$. Now from (1) there exists $z\in A_{j}$ such that $j(x)\wedge z=j(0)$ and $j(x)\vee z=1$. Since $a\lessdot j(x)$ and $j(0)\leq a$ then $z\leq a$. Hence, 
\[j(x)=j(x)\vee a\geq j(x)\vee z=1.\]
Thus $j(b\wedge x)=j(b)\wedge j(x)=j(b)\geq b$, that is, $[b\wedge x,b]\in\EuScript{D}_{j}$.

Reciprocally, note that the interval $[j(0),1]$ satisfies $j\leq\chi(j(0),1)$. By hypothesis, for every $a\in A_{j}$ there exists $b\in[j(0),1]$ such that $j(0)=a\wedge b=a\wedge j(b)$ and $[a\vee b, 1]\in\EuScript{D}_{j}$, this is equivalent to $1=a\vee_{j} j(b)$ (where $\vee_{j}$ is the supremum in $A_{j}$). Therefore every element in $A_{j}$ has a complement as required.
\end{proof}




Next we will see an important property of the negation of a spectral nucleus. The following Proposition uses the Lemma \ref{0000}.

\begin{prop}\label{lolya}

Let $j$ be an spectral nucleus. For every $[a,b]\in\EuScript{D}_{\neg j}=\neg\EuScript{D}_{j}$ the following assertions hold:
\begin{itemize}
\item[(1)] $[a,b]\in\EuScript{C}mp(\EuScript{O})$.
\item[(2)] There exists $y\in [a,b]$ such that $[y,b]\in\EuScript{S}ct(\EuScript{D}_{j})$.
\end{itemize}
\end{prop}

\begin{proof}

Consider any interval $[a,b]\in\EuScript{D}_{\neg j}=\neg\EuScript{D}_{j}$. Observe that $j\leq\chi(a,b)$ since $[a, j(a)\wedge b]\in\EuScript{D}_{j}\cap\neg\EuScript{D}_{j}=\EuScript{O}$, hence $[a,b]\in\EuScript{F}ll(\EuScript{D}_{j})$  by Proposition \ref{lolyu}. 

For (1), consider any $a\leq x\leq b$ then there exists $a\leq y\leq b$ such that $a=x\wedge y$ and $[x\vee y,b]\in\EuScript{D}_{j}$ also this interval is a sub-interval of $[a,b]$ then  $[x\vee y,b]\in\EuScript{D}_{j}\cap\neg\EuScript{D}_{j}=\EuScript{O}$ that is, $x\vee y=b$.

Now for assertion (2) consider $a\leq b\wedge soc_{j}^{\infty}(a)\leq b$. For this element we can find $a\leq y \leq b$ such that $a=y\wedge b\wedge soc_{j}^{\infty}(a)=y\wedge soc_{j}^{\infty}(a)$ and $[y\vee (b\wedge soc_{j}^{\infty}(a)),b]\in\EuScript{D}_{j}$. Let us see that the element $y$ satisfies our requirements. First this interval is a sub-interval of $[a,b]$ thus $y\vee (b\wedge soc_{j}^{\infty}(a))=b$ and by modularity we have that $b\wedge(y\vee soc_{j}^{\infty}(a))=b$, that is, $b\leq y\vee soc_{j}^{\infty}(a)$. Therefore $b\leq soc_{j}^{\infty}(y)$ which is equivalent to $[y,b]\in\EuScript{G}ab(\EuScript{D}_{j})$. By Lemma 6.4 of \cite{simmonsrelative} we have $[y,b]\in\EuScript{S}ct(\EuScript{D}_{j})$, as required.
\end{proof}

A consequence of this Proposition is the following:

\begin{cor}\label{rebsss}

Let $j\in\EuScript{E}(A)$. Then \[\EuScript{B}oy(\EuScript{D}_{\neg j})=\EuScript{B}oy(\EuScript{C}mp).\] which is equivalent to \[cdb^{\infty}_{\neg j}=cbd^{\infty}.\]

\end{cor}

\begin{proof}
This is immediate from the fact that $\EuScript{D}_{\neg j}\subseteq\EuScript{C}mp.$
\end{proof}

\begin{dfn}\label{WA}
An interval $[a,b]$ on an idiom $A$ is \emph{weakly atomic}, if for every $a\leq c\leq d\leq b$ there exists $c\leq x<y\leq d$ with $[x,y]\in\EuScript{S}mp$. Denote by $\EuScript{W}\EuScript{A}$ the set of all weakly atomic intervals.
\end{dfn}

In \cite{simmonscantor} it is proved that $\EuScript{W}\EuScript{A}$ is a division se, moreover, this set has the property: \[\EuScript{C}mp\cap\EuScript{W}\EuScript{A}=\EuScript{S}\EuScript{S}p\] Theorem 7.9 in \cite{simmonscantor}.

\begin{cor}\label{rebb}
Let $A$ be an idiom and $j\in\EuScript{E}(A)$ then every $[a,b]\in\EuScript{D}_{\neg j}\cap\EuScript{W}\EuScript{A}$ is semi-simple.
\end{cor}
\begin{proof}
From Proposition \ref{lolya}.(1), we have that any $[a,b]\in\EuScript{D}_{\neg j}\cap\EuScript{W}\EuScript{A}$ is complemented. Therefore $\EuScript{D}_{\neg j}\cap\EuScript{W}\EuScript{A}\subseteq\EuScript{C}mp\cap\EuScript{W}\EuScript{A}=\EuScript{S}\EuScript{S}p$. 

\end{proof}

\begin{cor}\label{reb}
Let $A$ be a compactly generated idiom, and $j\in\EuScript{E}(A)$, then any interval $[a,b]\in\EuScript{D}_{\neg j}$ is semi-simple.
\end{cor}

\begin{proof}
If $A$ is compactly generated then is weakly atomic (see Theorem 7.8 of \cite{simmonscantor}) therefore this is a direct consequence of Proposition \ref{rebb}.
\end{proof}

The lattice $\Lambda(M)$ is compactly generated for every module $M$, thus the above facts resemble the module theoretic environment.

The following Theorem is the idiomatic analogue of Lemma 2.12 in \cite{arroyo1994some}. 

\begin{thm}\label{rebab}
Let $j$ be a nucleus on $A$ such that $Boy(j)=tp$. Then $\neg j\vee \neg\neg j=tp$.
\end{thm}

\begin{proof}
Recall that $N(A)_{\neg\neg}$ is a complete boolean algebra then the element 
\[\neg\neg j\in N(A)_{\neg\neg}\] 
has a unique complement there and this complement is $\neg j$. Recall that the suprema in the quotient $N(A)_{\neg\neg}$ is describe as: \[\neg\neg(\neg\neg k\vee\neg\neg k')\] for any $k, k'\in N(A)$ and in our case we have \[\neg\neg(\neg\neg j\vee\neg j)=tp\]

Now under the hypothesis $[j, Boy(j)]=[j, tp]$ is a complete boolean algebra, then for the nucleus $\neg\neg j\vee\neg j$ there exists $l\in [j, tp]$ such that $(\neg\neg j\vee\neg j)\vee l=tp$ and $(\neg\neg j\vee\neg j)\wedge l=j$. Applying $\neg\neg(\_)$ to the last equality
\[\neg\neg(\neg\neg j\vee\neg j)\wedge\neg\neg l=tp\wedge\neg\neg l=\neg\neg l=\neg\neg j\]
and using $(\neg\neg j\vee\neg j)\vee l=tp$ we have 
\[tp=(\neg\neg j\vee\neg j)\vee l\leq(\neg\neg j\vee\neg j)\vee\neg\neg l=(\neg\neg j\vee\neg j)\vee\neg\neg j=\neg\neg j\vee\neg j=tp\] 
as required.

\end{proof}

As a consequence of the above we have:
\begin{cor}\label{rebt}
Let $j\in\EuScript{E}(A)$ then $\neg j\vee \neg\neg j=tp$.
\end{cor}

\begin{cor}\label{rebs}
Let $j$ be a nucleus such that $j, \neg j\in\EuScript{E}(A)$ then the intervals of \[\EuScript{D}_{\neg j}\cap\EuScript{W}\EuScript{A}\] and \[\EuScript{D}_{\neg\neg j}\cap\EuScript{W}\EuScript{A}\] are semi-atomic.
\end{cor}

\begin{proof}

By Proposition \ref{lolya}.(1) we have \[\EuScript{D}_{\neg j}\cap\EuScript{W}\EuScript{A}\subseteq\EuScript{W}\EuScript{A}\cap\EuScript{C}mp=\EuScript{S}\EuScript{S}p\] and by the spectral property of $\neg j$ we also have \[\EuScript{D}_{\neg\neg j}\cap\EuScript{W}\EuScript{A}\subseteq\EuScript{W}\EuScript{A}\cap\EuScript{C}mp=\EuScript{S}\EuScript{S}p\] then \[(\EuScript{D}_{\neg j}\cup\EuScript{D}_{\neg\neg j})\cap\EuScript{W}\EuScript{A}\subseteq\EuScript{S}\EuScript{S}p\] using the frame distributivity law of $\EuScript{B}(A)$. Taking division sets we have  \[\EuScript{D}vs(\EuScript{D}_{\neg j}\cup\EuScript{D}_{\neg\neg j})\cap\EuScript{W}\EuScript{A}=\EuScript{W}\EuScript{A}\subseteq\EuScript{D}vs(\EuScript{S}\EuScript{S}p)=\EuScript{S}\EuScript{A}\] the first equality follows from Theorem \ref{rebab} and the second equality from Theorem 7.11 of \cite{simmonscantor}.
\end{proof}

\begin{cor}\label{rebss}
Let $A$ be a weakly atomic idiom, that is, $\EuScript{W}\EuScript{A}=\EuScript{I}(A)$ then $A$ is semi-atomic, that is, $Gab(id)=soc^{\infty}=tp$. 
\end{cor}

\begin{proof}
This follows directly from Corollary \ref{rebs}.
\end{proof}
The following appear as Lemma 2.13 of \cite{arroyo1994some}.
\begin{cor}\label{corr}
Let $R$ be a ring and supuse that $j, \neg j\in\EuScript{E}(\Xi(R))$ then $R$ is left semiartian ring.
\end{cor}
\begin{proof}
Direct from \ref{rebss}.
\end{proof}

\bibliographystyle{amsalpha}

\bibliography{researchupdate}

\end{document}